\documentclass[12pt,reqno]{amsart}

\usepackage[T1]{fontenc}
\pdfoutput=1
\usepackage{mathptmx,microtype}
\usepackage{amsthm,amsmath,amssymb}
\usepackage{enumitem}
\usepackage[font=footnotesize,margin=2.6cm]{caption}
\usepackage{cite}
\usepackage[hidelinks]{hyperref}

\theoremstyle{plain}
\newtheorem{thm}{Theorem}
\newtheorem{lemma}[thm]{Lemma}
\newtheorem{corollary}[thm]{Corollary}
\newtheorem{prop}[thm]{Proposition}

\newcommand\cT{{\mathcal T}}
\newcommand\cP{{\mathcal P}}

\renewcommand{\P}{{\mathbb P}}
\newcommand\E{{\mathbb E}}
\newcommand{\eps}{\varepsilon}
\newcommand{\I}[1]{\mathbf{1}_{\{#1\}}}

\renewcommand{\le}{\leqslant}
\renewcommand{\ge}{\geqslant}

\author[M. Bassan]{Michal Bassan}
\address{University of Oxford, 
Department of Statistics and Keble College}
\email{michal.bassan@keble.ox.ac.uk}

\author[S. Donderwinkel]{Serte Donderwinkel}
\address{University of Groningen, 
Bernoulli Institute for Mathematics, Computer Science and AI, 
and CogniGron (Groningen Cognitive Systems and Materials Center)}
\email{s.a.donderwinkel@rug.nl}

\author[B. Kolesnik]{Brett Kolesnik}
\address{University of Warwick, 
Department of Statistics}
\email{brett.kolesnik@warwick.ac.uk}

	% Infinitely divisible distributions

\begin{document}

\title[To see the forest for the trees]
{To see the forest for the trees: On the infinite divisibility 
of unlabeled forests}

%%%%%%%%%%%%%%%%%%%%%%%%%%%%
%%%%%%%%%%%%%%%%%%%%%%%%%%%%
%%%%%%%%%%%%%%%%%%%%%%%%%%%%
%%%%%%%%%%%%%%%%%%%%%%%%%%%%
%%%%%%%%%%%%%%%%%%%%%%%%%%%%
\begin{abstract} 
Inspired by Stufler's  recent probabilistic proof of Otter's 
asymptotic number of unlabeled trees, we revisit work of 
Palmer and Schwenk, and study unlabeled forests from 
a probabilistic point of view. We show that the number of 
trees in a random forest converges, with all of its moments, 
to a shifted compound Poisson. We also find the asymptotic 
proportion of forests that are trees. The key fact is that the 
number of trees $t_n$ and forests $f_n$ are related by a 
L\'evy process. As such, the results by Palmer and Schwenk 
follow by an earlier and far-reaching limit theory by Hawkes 
and Jenkins. We also show how this limit theory implies  
results by Schwenk and by Meir and Moon, related to 
degrees in large random trees. Our arguments apply, 
more generally, to the enumeration of sub-exponentially 
weighted integer partitions, or, in fact, any setting where 
the underlying L\'evy process follows the 
one big jump principle. 
\end{abstract}

\maketitle

%%%%%%%%%%%%%%%%%%%%%%%%%%%%
%%%%%%%%%%%%%%%%%%%%%%%%%%%%
%%%%%%%%%%%%%%%%%%%%%%%%%%%%
%%%%%%%%%%%%%%%%%%%%%%%%%%%%
%%%%%%%%%%%%%%%%%%%%%%%%%%%%
\section{Introduction}

Stufler \cite{Stu23} recently gave a probabilistic proof of 
Otter's \cite{Ott48}  
result on
the asymptotic 
number $t_n$ of unlabeled trees 
on $n$ vertices, showing   
that 
\begin{equation}\label{E_Ott}
\lim_{n\to\infty}
\alpha^n n^{5/2} t_n
=\beta,
\end{equation}
where
$\alpha$ and $\beta$ are called {\it Otter's constants}. 
We note
that 
$\alpha\approx 0.338$ and $\beta\approx 0.534$; 
see, e.g., Finch \cite{Fin03}. 
The first constant satisfies  
\begin{equation}\label{E_alpha}
\alpha
=
\prod_{k=1}^\infty (1-\alpha^k)^{r_k},
\end{equation}
where $r_n$ is the number 
of {\it rooted} 
unlabeled trees on $n$ vertices. 
The second constant is related to 
the average preimage size of the 
surjection 
$(T,v)\mapsto R$, 
where $T$ is an unlabeled tree, 
$v$ is one of its $n$ vertices, 
and $R$ is $T$ rooted at $v$. 
Specifically, 
\begin{equation}\label{E_Ott_r}
\lim_{n\to \infty}
\frac{nt_n}{r_n}
= 
(2\pi \beta^2)^{1/3}. 
\end{equation}

A number of years later,  
Palmer and Schwenk \cite{PS79}
studied asymptotic properties of 
the   
number $f_n$ of 
unlabeled forests on $n$ vertices,  
using singularity analysis. 
In this work, we take a probabilistic 
point of view.  
In particular, in conjunction 
with Stufler's \cite{Stu23} proof of 
\eqref{E_Ott}, we obtain
alternative, probabilistic
proofs of the main results
in \cite{PS79}, 
stated as \eqref{E_PS2}
and \eqref{E_PS1} below. 
Rather than using singularity analysis, 
we will employ a 
Tauberian-type theorem by 
Hawkes and Jenkins \cite{HJ78}, 
based on the 
L\'evy--Khintchine formula from the 
theory of L\'evy proceses. 

In fact, our methods   
are more widely applicable, reaching well beyond  
trees and forests; see Section
\ref{S_apps} below.
We will, however, focus on this specific case
of interest in order to bring our probabilistic 
techniques and intuition into the clearest possible
light.

Our first result is as follows.

Let $t(x)=\sum_{n=1}^\infty t_n x^n$
be the generating function for $t_n$, 
starting the sum at $n=1$
for convenience.

We recall (see, e.g., 
Aldous \cite{Ald89})
that a random variable $X$ is {\it compound Poisson} 
distributed, with 
{\it  compounding measure} $\xi$,  
if 
\begin{equation}\label{E_CP}
\E (e^{-\phi X})
=\exp\left(
-\int_0^\infty(1-e^{-\phi x})\xi(dx)
\right),
\end{equation}
for all $\phi>0$.

\begin{thm}[Tree-forest distribution]
\label{T_TFd}
Let $\cT_n$ be the number of trees
in a uniformly random unlabeled forest on $n$ vertices. 
Then, as $n\to\infty$, $\cT_n-1$
converges, with 
all of its moments, 
to a discrete compound Poisson random variable $\cP$, 
with compounding measure  
 $\xi_k=t(\alpha^k)/k$ on integers $k\ge1$. 
\end{thm} 

Generalizations of this 
limit theorem appear in, e.g., 
Bell, Bender, Cameron  and Richmond \cite{BBCR00}
and Stufler \cite{Stu20}, 
but not with convergence of
all moments. The convergence in distribution in the 
theorem implies that the number of trees in 
the forest is tight as its size grows. 
We strengthen this and show that all the 
powers of the number of trees are uniformly integrable. 
Both of these facts 
are an effect of the  
well-known
{\it one big jump principle} that implies that 
a forest of large size is most likely big 
because of one big tree, 
a phenomenon sometimes 
referred to as \emph{condensation}.

A more general version of this result, 
with essentially the same proof, 
carries forward in any situation where $t_n$ 
is replaced by some other 
$x_n\sim b a^{-n} n^{-c}$, with $c>1$; 
or, even more generally, 
when $a^nx_n$
is proportional to 
a {\it sub-exponential} 
probability distribution
(so that the one big jump principle applies). 
See Section \ref{S_OBJP} and \ref{S_apps}
below for definitions and more context. 

In particular, we recover 
Theorem 2 in \cite{PS79},
stating that 
\begin{equation}\label{E_PS2}
\lim_{n\to\infty}
\E(\cT_n)\to 1+\sum_{k=1}^\infty t(\alpha^k). 
\end{equation}

In fact, by Theorem \ref{T_TFd}, 
 and reversing the order of summation, 
we obtain the following result. 

\begin{corollary}\label{C:TinF}
As $n\to\infty$, 
\[
\E(\cT_n)\to 
1+\sum_{k=1}^\infty \frac{t_k}{1/\alpha^k-1}
\]
and
\[
{\rm Var}(\cT_n)\to 
\sum_{k=1}^\infty \frac{k t_k}{1/\alpha^k-1}.
\]
\end{corollary}

Using the values for $t_n$, $1\le n\le 100$, 
available at the OEIS \cite{A000055}, 
we note that the asymptotic mean and variance
of $\cT_n$ are approximately $1.755$ 
and $1.471$, respectively. 

Higher moments of $\cT_n$
are more complicated, 
expressible in terms of Bell polynomials. The $m$th
asymptotic cumulant 
of $\cT_n$, on the other hand, is simply 
\[
\lim_{n\to\infty}
\kappa_m(\cT_n)= 
\sum_{k=1}^\infty \frac{k^{m-1} t_k}{1/\alpha^k-1}, 
\]
for all $m>1$. 
See, e.g., 
Mane \cite{Man24}.

In proving Theorem \ref{T_TFd}, 
we will also give an 
alternative, probabilistic proof 
of Theorem 1 in \cite{PS79},
concerning 
the asymptotic proportion of 
unlabeled forests that are connected, 
showing that  
\begin{equation}\label{E_PS1}
\lim_{n\to\infty}
\frac{t_n}{f_n}\to \frac{1}{1+f(\alpha)},
\end{equation}
where $f(x)=\sum_{i=1}^\infty f_n x^n$
is the generating function for $f_n$;
see Corollary \ref{T_TF} below.

%%%%%%%%%%%%%%%%%%%%%%%%%%%%
%%%%%%%%%%%%%%%%%%%%%%%%%%%%
%%%%%%%%%%%%%%%%%%%%%%%%%%%%
%%%%%%%%%%%%%%%%%%%%%%%%%%%%
%%%%%%%%%%%%%%%%%%%%%%%%%%%%
\subsection{Our approach}

We recall that a probability distribution $\pi$ 
is {\it infinitely divisible}
if, for any $n\ge0$, 
there are 
independent and identically distributed 
$X_1,\ldots,X_n$ such that $X_1+\cdots+X_n\sim \pi$
has this distribution. 

The following
observation 
that all three sequences  
$t_n$, $r_n$ and $f_n$
are related by an infinitely divisible 
probability distribution is central to our arguments. 
We use the convention that $f_0=1$. 

\begin{thm}[Otter's distribution]
\label{T_OD}
Put 
\[
\rho=\prod_{k=1}^\infty (1-\alpha^k)^{t_k}. 
\]
Then the sequence
$(p_n,n\ge0)$ with 
\begin{equation}\label{E_pn}
p_n 
= \rho\alpha^n f_n
= f_n
\prod_{k=1}^\infty(1-\alpha^k)^{t_k+nr_k}
\end{equation}
is an {\it infinitely divisible} probability distribution. 
\end{thm}
In fact, the proof of this theorem shows that,
for any $0<\hat \alpha\le \alpha$, 
\begin{equation}\label{E_hatpn}
\hat p_n=\hat \alpha^n f_n \prod_{k=1}^\infty (1-\hat \alpha^k)^{t_k}
\end{equation}
is an infinitely divisible probability distribution,
so we in fact obtain a family of such distributions. 
The second equality in \eqref{E_pn}, however, 
holds only when $\hat \alpha=\alpha$, using 
\eqref{E_alpha}.

Each infinitely divisible 
probability distribution 
$\pi$   
has an associated {\it L\'evy process} $(L_t, t\ge0)$, 
whose state $L_1\sim\pi$ at time $t=1$
has this distribution. 
The {\it L\'evy--Khintchine formula} relates $\pi$
with a corresponding 
{\it L\'evy measure} $\nu$, 
which controls the jumps of the process. 
When $(\pi_n,n\ge0)$ 
is supported on the non-negative integers, 
this formula states that 
$(\nu_j,j\ge1)$ satisfies  
\begin{equation}\label{E_LK}
\sum_{n=0}^\infty \pi_n x^n
=\exp\left(
-\sum_{j=1}^\infty (1-x^j)\nu_j
\right), 
\end{equation}
and thus $\pi$ is a compound 
Poisson distribution with compounding measure $\nu$   
(see\ \eqref{E_CP} above).
In this case,  the \emph{total L\'evy measure} 
$\lambda=\sum_{j=1}^\infty 
\nu_j$ 
is finite. 
Hence $\mu_n=\nu_n/\lambda$
is a well-defined 
probability measure and $\pi$ is the 
law of the sum of $\mathrm{Poisson}(\lambda)$ 
independent samples from $\mu$. 

As we will see (see Section \ref{S_Ottd} below), 
in the specific case of 
Otter's distribution $\pi_n=p_n$, 
the Lévy measure $\nu_n$ then equals 
\begin{equation}\label{E_nu}
\tau_n = \frac{\alpha^n}{n}\sum_{d|n}dt_n,
\end{equation}
so that by changing the order of summation 
(see (\ref{E_rev}) below), we find that  
$\rho=e^{-\lambda}$. 
By \eqref{E_Ott}, we find that $\tau_n\sim\alpha^nt_n$ is 
{\it regularly varying} with index $\gamma=-5/2$, 
in the sense that 
$\tau_{\lfloor xn\rfloor} \sim x^{\gamma} \tau_n$, 
for all $x>0$
(see, e.g., Bojanic and Seneta \cite{BS73}).
A general result of Hawkes and Jenkins \cite{HJ78}, 
from a year before \cite{PS79}, states that 
if $\pi_n$ and $\nu_n$
are related by \eqref{E_LK}
and $\nu_n$ is regularly varying 
with index $\gamma<-1$, 
then $\pi_n\sim\nu_n$. 
As such, we obtain the following result, 
which, as we will see 
(see again Section \ref{S_Ottd}),  
implies \eqref{E_PS1}.

\begin{corollary}
\label{T_TF}
As $n\to\infty$, we have that
$f_n\sim t_n/\rho$. 
\end{corollary}

We note that  
$p_n$ is also regularly varying. 
On the other hand, the 
infinitely divisible 
distributions 
$\hat p_n$ in \eqref{E_hatpn}
are not regularly varying when
$\hat \alpha<\alpha$.

\subsection{One big jump principle}
\label{S_OBJP}

Our proof of Corollary \ref{T_TF}
is based on the so-called  
{\it one big jump principle;}
see, e.g., 
\cite{Chi64,Hey68,Nag69}. 

The result by 
Hawkes and Jenkins \cite{HJ78}, 
as discussed above, 
was later 
extended by Embrechts and Hawkes \cite{EH82}, 
who showed that 
the following are equivalent, 
as $n\to\infty$:
\begin{enumerate}
\item[(i)] $\mu_n^*\sim 2\mu_n$ and $\mu_n\sim\mu_{n+1}$; 
\item[(ii)] $\pi_n^*\sim 2\pi_n$ and $\pi_n\sim \pi_{n+1}$;
\item[(iii)] $\pi_n\sim\nu_n$ and $\nu_{n}\sim\nu_{n+1}$. 
\end{enumerate}
Above, as usual, $\mu^*_n$ (and, similarly, $\pi_n^*$)
denotes the {\it self-convolution} 
\[
\mu_n^*=\sum_{k=0}^{n} \mu_k \mu_{n-k}.
\]

A probability measure $(\mu_n,n\ge0)$ is 
{\it sub-exponential} 
if condition (i) above is satisfied. 
Such distributions satisfy the 
so-called 
one big jump principle;
namely, 
if the sum of two (or more) independent 
sub-exponential random variables 
takes a large value, 
then most likely one of them has taken (essentially)
this value. 

The equivalent condition (iii) then says that the 
probability $\pi_n$ that the L\'evy process 
$L_1\sim \pi$ takes a 
large value $n$ is asymptotically equivalent to $\nu_n$. 
For $j\ge 1$, the value $\nu_j$ is the expected number of 
jumps by $L$ of size $j$ by time $t=1$ (i.e., 
times $s\le 1$ that 
$\lim_{\eps\downarrow0}L_{s-\eps}=L_s-j$). 
For large $j$ this is close to the probability of having 
one jump of size $j$, so $\pi_n\sim \nu_n$ essentially 
says that if $L_1$ takes a large value, 
it is most likely due to one large jump.

Since, by \eqref{E_Ott}, 
$\alpha^n t_n$ is regularly
varying with index $\gamma=-5/2$, 
it is natural to expect, 
in light of 
\eqref{E_pn} and 
\eqref{E_nu}, 
the asymptotics in 
Corollary \ref{T_TF} to hold. 

In closing, let us mention 
that different versions of the 
one big jump principle have been used recently by 
Stufler \cite{Stu20} to study the sizes of the small structures 
in a random multiset of large total size, and by 
Panagioutou and Ramzews \cite{PR24} to study 
the rare event that a  multiset of large total size consists 
of many objects. Both 
of these works apply, 
in particular, to 
our current case of 
trees in large forests.

%%%%%%%%%%%%%%%%%%%%%%%%%%%%
%%%%%%%%%%%%%%%%%%%%%%%%%%%%
%%%%%%%%%%%%%%%%%%%%%%%%%%%%
%%%%%%%%%%%%%%%%%%%%%%%%%%%%
%%%%%%%%%%%%%%%%%%%%%%%%%%%%
\subsection{Further applications}
\label{S_apps}
We have focused on  
unlabeled trees and forests; however, 
our methods 
are more widely applicable. 
In this way, we provide a probabilistic
alternative to previous works, 
such as   
Compton \cite{Com87}
and 
Bell, Bender, Cameron and Richmond \cite{BBCR00}, 
which use singularity analysis to study 
distributional
features 
of large random combinatorial structures. 

For instance, our arguments extend
 immediately to the setting of {\it weighted 
integer partitions} (multi-sets)
with sub-exponential weights. 
In fact, the potential application of this 
limit theory to partitions and related structures
is alluded to in \cite[pp.\ 66--67]{HJ78}.
In the case of unlabeled trees and forests, 
we effectively assign 
weight $t_k$ to 
each integer $k\ge 1$, so that there are a total of 
$f_n$ weighted partitions of $n$. 
Of course, other weights
are possible. 
See again 
\cite{Com87,BBCR00}
and, e.g., 
the more recent works 
\cite{BG05,Mut11,Stu20,PR24}.

In another direction, see 
our recent applications 
to the enumeration of
{\it tournament score sequences} \cite{Kol23,BDK24b},
{\it Sina\u{\i} excursions} \cite{DK24b}  
 and 
{\it graphical sequences} \cite{BDK24a}.
In all of these cases, 
in contrast to the current work, the
combinatorial objects
of interest are {\it ordered}.
Indeed, the first two  
are examples of {\it renewal sequences,}
which  
can be thought of as 
{\it weighted integer compositions}. 
Graphical sequences, 
on the other hand, are {\it delayed} renewal sequences
(also called {\it Riordan arrays} in combinatorics).

We expect many further combinatorial 
applications to be found. 
The limit theory 
\cite{HJ78,EH82}
is based
on work by 
Chover, Ney and Wainger \cite{CNW73},  
which studies analytic transformations
of probability measures, 
vastly generalizing 
the classical renewal theorem. 
As such, our current methods
can,  subject to the conditions 
in \cite{HJ78,EH82}, be applied 
when a 
combinatorial structure 
decomposes  
into smaller irreducible parts, 
as is often the case. 

Finally, let us highlight the 
vast 
generalization of the 
main limit theorems in \cite{HJ78,EH82}, 
stated as Theorem C in  
Embrechts and Omey \cite{EO84}, 
where $\exp(z)$
in \eqref{E_LK} is replaced with
some other analytic function $\phi(z)$, 
pointing to applications awaiting, 
well beyond the 
L\'evy--Khintchine correspondence.

%%%%%%%%%%%%%%%%%%%%%%%%%%%%
%%%%%%%%%%%%%%%%%%%%%%%%%%%%
%%%%%%%%%%%%%%%%%%%%%%%%%%%%
%%%%%%%%%%%%%%%%%%%%%%%%%%%%
%%%%%%%%%%%%%%%%%%%%%%%%%%%%
\subsection{Outline}

In Section \ref{S_Ottd}, 
we prove Theorem \ref{T_OD}
and its Corollary \ref{T_TF}. 
In Section \ref{S_TFd}, 
we prove our main result, 
Theorem \ref{T_TFd}.
Finally, in Section \ref{S_final},
we discuss some further
applications, 
related to 
cycle index asymptotics
and 
degrees
in large random trees, 
giving an alternative 
proof of a result of 
Schwenk \cite{Sch77} 
(cf.\ 
Meir and Moon \cite{MM83}).

%%%%%%%%%%%%%%%%%%%%%%%%%%%%
%%%%%%%%%%%%%%%%%%%%%%%%%%%%
%%%%%%%%%%%%%%%%%%%%%%%%%%%%
%%%%%%%%%%%%%%%%%%%%%%%%%%%%
%%%%%%%%%%%%%%%%%%%%%%%%%%%%
\subsection{Acknowledgments}
MB is supported by a Clarendon Fund Scholarship. 
SD acknowledges
the financial support of the CogniGron research center
and the Ubbo Emmius Funds (University of Groningen). 
Her research was also partially supported by the 
Marie Sk\l{}odowska-Curie grant GraPhTra 
(Universality in phase transitions in random graphs), 
grant agreement ID 101211705.
Part of this research was performed while MB and SD 
were visiting the Mathematical Sciences Research Institute (MSRI), 
now becoming the Simons Laufer Mathematical Sciences Institute (SLMath), 
which is supported by the 
National Science Foundation (Grant No.\ DMS-1928930).

%%%%%%%%%%%%%%%%%%%%%%%%%%%%
%%%%%%%%%%%%%%%%%%%%%%%%%%%%
%%%%%%%%%%%%%%%%%%%%%%%%%%%%
%%%%%%%%%%%%%%%%%%%%%%%%%%%%
%%%%%%%%%%%%%%%%%%%%%%%%%%%%
\section{Otter's distribution}
\label{S_Ottd} 

A forest is a multi-set of trees. 
As such, their generating functions 
$t(x)$ and $f(x)$
are related via
{\it P\'{o}lya's exponential,}
as can be seen using the {\it log-exp
transform,} as a simple instance 
of the {\it symbolic method} in 
Flajolet and Sedgewick 
\cite[p.\ 29]{FS09}. 
Indeed, as explained therein, 
let $\cT$ be the set of all unlabeled trees. 
For each $T\in\cT$,
let $|T|\ge1$ be its number of vertices. 
Each forest can be expressed as a sequence
of trees, in some canonical order.
As such, it follows that 
$1+f(x)$ 
(with $1$ accounting for the 
$f_0=1$ empty forest containing no trees) 
is an infinite product of geometric series, 
with one for each $T\in \cT$:
\begin{align}
1+f(x)&=\prod_{T\in\cT}\frac{1}{1-x^{|T|}}
=\prod_{k=1}^\infty\left(\frac{1}{1-x^k}\right)^{t_k}\nonumber\\
&=\exp\left(-\sum_{k=1}^\infty t_k\log(1-x^k)
\right)
=\exp\left(\sum_{k=1}^\infty \frac{t(x^k)}{k}
\right).\label{E_Polya2}
\end{align}
Technically, since $\cT$ is infinite, a limiting procedure
is required in order to make the above precise, 
but this is the basic idea.

The coefficient of $x^n$
inside the exponential, on the right
hand side of \eqref{E_Polya2}, is 
given by 
\begin{equation}\label{E_coeff}
[x^n]\sum_{k=1}^\infty \frac{t(x^k)}{k}
=\frac{1}{n}
\sum_{d|n} d t_d. 
\end{equation}
Therefore, the L\'evy--Khintchine formula \eqref{E_LK} 
is satisfied, with  
infinitely divisible probability distribution 
\[
p_n=\alpha^n f_n e^{-\lambda} 
\]
and L\'evy measure 
\begin{equation}\label{E_asynun}
\tau_n=\frac{\alpha^n}{n}\sum_{d|n} d t_d\sim \alpha^nt_n, 
\end{equation}
which, upon interchanging the 
order of summation,  
can be seen to have total measure 
\begin{equation}\label{E_rev}
\lambda
=
\sum_{n=1}^\infty 
\frac{\alpha^n}{n} 
\sum_{d|n} d t_d 
=\sum_{k=1}^\infty 
t_k
\sum_{i=1}^\infty 
\frac{\alpha^{ki}}{i} \\
=-\sum_{k=1}^\infty t_k\log(1-\alpha^k). 
\end{equation}
Noting that 
\[
e^{-\lambda}
=\prod_{k=1}^\infty(1-\alpha^k)^{t_k}=\rho,
\]
Theorem \ref{T_OD} follows. 

Finally, as discussed above, 
by applying \cite{HJ78}
we find that 
$p_n\sim\tau_n$
and Corollary \ref{T_TF}
follows. 
In light of \eqref{E_Polya2}, 
this is equivalent to \eqref{E_PS1}.

%%%%%%%%%%%%%%%%%%%%%%%%%%%%
%%%%%%%%%%%%%%%%%%%%%%%%%%%%
%%%%%%%%%%%%%%%%%%%%%%%%%%%%
%%%%%%%%%%%%%%%%%%%%%%%%%%%%
%%%%%%%%%%%%%%%%%%%%%%%%%%%%
\section{Tree-forest distribution}
\label{S_TFd}

We recall that 
the {\it cycle
index}  $Z(S_n)$
of the symmetric group $S_n$
is the polynomial 
in the variables $x_1,\ldots,x_n$
defined by 
\[
Z(S_n)=Z(S_n;x_1,\ldots,x_n)
=
\frac{1}{n!}
\sum_{\sigma\in S_n}
\prod_{\ell=1}^n 
x_\ell^{c_\ell(\sigma)},
\]
where $c_\ell(\sigma)$
is the number of cycles
in the permutation $\sigma\in S_n$ of length $\ell$. 
It can be computed recursively. 
We note that $Z_0=1$
and
\begin{equation}\label{E_Zrec}
kZ(S_k)=\sum_{i=1}^k x_i Z(S_{k-i}),
\quad k\ge 1. 
\end{equation}
which can be deduced by 
considering the cycle of $1$ in a $\sigma\in S_k$. 
By classical 
P\'olya theory (see, e.g., \cite[Eq.\ (91), p.\
 86]{FS09}) for any $k\ge 1$, 
 the generating function for the 
 number of unlabeled forests with $k$ trees equals 
 $Z(S_{k},t(x),\dots, t(x^{k}))$, 
 so the number of unlabeled forests of size $n$ 
 with exactly $k+1$ trees equals
\[
[x^n]Z(S_{k+1},t(x),\dots, t(x^{k+1})).
\]
By Robinson and Schwenk \cite{RS75}, 
this is asymptotically equivalent to $t_n\zeta_k$ 
as $n\to \infty$, with $\zeta_1=1$ 
and 
\begin{equation}\label{E_zeta}
\zeta_k=Z(S_{k},t(\alpha),\dots, t(\alpha^{k})),\quad k\ge 1. 
\end{equation} 
As observed in \cite[p.\ 120]{PS79}, 
dividing by $f_n$ and 
using its asymptotics in Corollary \ref{T_TF}, 
the 
probability 
that a random unlabeled 
forest on $n$ vertices has 
exactly $k+1$ trees converges to
\[
\gamma_k:=\rho \zeta_k.
\]
as $n\to\infty$.

Moreover, by \eqref{E_Zrec},  
the generating functions
$\zeta(x)=\sum_{k=0}^\infty \zeta_kx^k$ 
and $\xi(x)=\sum_{i=1}^\infty t(\alpha^i)x^i$
satisfy $x\zeta'(x)=\xi(x)\zeta(x)$. 
By integrating, 
we find that 
\begin{equation}\label{E_LKZ}
\sum_{k=0}^\infty \zeta_kx^k
=\exp\left(\sum_{i=1}^\infty 
\frac{t(\alpha^i)}{i}x^i
\right).
\end{equation}

Comparing \eqref{E_LKZ} with
\eqref{E_LK} and \eqref{E_Polya2}, 
it follows that the 
sequence $(\gamma_k,k\ge0)$
is an infinitely divisible probability distribution, 
with L\'evy measure
$(\nu_i,i\ge1)$ given by 
$\nu_i=t(\alpha^i)/i$. 
Therefore, we obtain the 
following result, 
by the discrete compound Poisson characterization 
of  
infinitely divisible probability distributions on the non-negative
integers; see, e.g., Feller 
\cite[p.\ 290]{Fel68}.

\begin{prop}
\label{P_TFd}
Let $\cT_n$ be the number of trees
in a uniformly random unlabeled forest on $n$ vertices. 
Then, as $n\to\infty$, $\cT_n-1$
converges  
to a discrete compound Poisson random variable $\cP$, 
with compounding measure  
$t(\alpha^k)/k$ 
on integers $k\ge1$. 
\end{prop}

In the next section we show that this  
probabilistic rephrasing
of the observations in 
\cite{PS79} 
yields a direct 
proof of \eqref{E_PS2}, 
that appears as Theorem 2 in 
in  \cite{PS79}, 
as well as an extension 
to the higher moments of $\cT_n$.

%%%%%%%%%%%%%%%%%%%%%%%%%%%%
%%%%%%%%%%%%%%%%%%%%%%%%%%%%
%%%%%%%%%%%%%%%%%%%%%%%%%%%%
%%%%%%%%%%%%%%%%%%%%%%%%%%%%
%%%%%%%%%%%%%%%%%%%%%%%%%%%%
\subsection{Convergence of moments}

With Proposition \ref{P_TFd} in hand, 
we finish the proof of Theorem \ref{T_TFd} by showing 
that $\cT_n$ and its powers are uniformly integrable. 

The proof is once again related to the 
one big jump principle, which recall states 
that a forest is most likely big because of one big tree. 
To bound the probability that there are many trees 
in a large forest, we will 
use the fact that such (atypical) forests can be partitioned  
into multiple large forests. 

First, we 
provide the details of the uniform integrability of $\cT_n$ 
and will then explain how to adapt the argument to higher powers. 

\begin{lemma}
We have that 
\[
\lim_{K\to \infty} \limsup_{n\to\infty}
\E[\cT_n\I{\cT_n>K}]
=0.  
\]
\end{lemma}
\begin{proof}
Fix $K$ and observe that 
\begin{equation}\label{E_1stM}
\E[\cT_n\I{\cT_n>K}]
=K\P(\cT_n>K)
+\sum_{k=K}^\infty \P(\cT_n>k).
\end{equation}
Since $\cT_n$ converges in distribution to $\cP+1$, it follows
that 
\[
K\P(\cT_n>K)\to K\P(\cP\ge K). 
\]
Since $\E(\cP)<\infty$, 
the right hand side 
can be made arbitrarily 
small by making 
$K$ large. 

We turn to the second term in \eqref{E_1stM}. 
First, we note, by \eqref{E_Ott} and Corollary~\ref{T_TF}, 
that there are constants $c,C>0$ such that 
\begin{equation}\label{bound_fn}
c  
\le \alpha^n n^{5/2} f_n
\le C,  
\end{equation}
for all $n\ge1$.
Next, we observe that, 
for any $n$ and $k$,  
and any forest with $n$ vertices and more than $k$ trees, 
we can decompose the forest into a pair of forests, such that:  
(1) their sizes sum to $n$, 
(2) the first has at least $\lfloor k/2\rfloor$ vertices,  
and (3) the second is at least as big as the first. 
As such, 
the number of forests with $n$ vertices and more than $k$ trees 
is at most 
\[
\sum_{\ell=\lfloor k/2 \rfloor }^{\lfloor n/2\rfloor }f_{\ell}f_{n-\ell}.
\]
Therefore, applying \eqref{bound_fn}, 
we find that 
$\P(\cT_n>k)$
is at most 
\[
\frac{c^2}{C}
\sum_{\ell=\lfloor k/2 \rfloor}^{\lfloor n/2\rfloor} 
\left(\frac{n}{\ell(n-\ell)}\right)^{5/2}
\le 
\frac{2^{5/2} c^2}{C}
\sum_{\ell=\lfloor k/2 \rfloor}^{\lceil n/2 \rceil} \frac{1}{\ell^{5/2}}
\le 
\frac{c'}{k^{3/2}}, 
\]
for some constant $c'$ 
not depending on $k$ and $n$. 
This is summable, so, for any $\eps>0$,  
we can choose $K$ large so that 
$\E[\cT_n\I{\cT_n>K}]< \eps $,  
for all large $n$. 
\end{proof}

Finally, we note that the proof 
above can be adapted 
to show that $\cT^m_n$
is uniformly integrable, for 
any integer $m\ge1$. 
The basic idea is to 
split a forest with more than $k$ trees into $m+1$ 
forests of size at least $\lfloor k/(m+1)\rfloor$ 
and repeat the arguments above. 
The adaptation is straightforward, 
so we omit the details. 
Together with Proposition \ref{P_TFd}, this yields
Theorem \ref{T_TFd}. 

To see that Theorem~\ref{T_TFd} implies  
\eqref{E_PS2} 
and Corollary \ref{C:TinF}, note that 
\[
\E(\cP)=\sum_{k=0}^\infty k\gamma_k
=\sum_{i=1}^\infty t(\alpha^i).  
\]
By similar reasoning
as \eqref{E_coeff}
and \eqref{E_rev}, 
this is equal to 
\[
\sum_{k=1}^\infty \alpha^k\sum_{d|k}t_d=
\sum_{k=1}^\infty \frac{t_k}{1/\alpha^k-1},
\]
in line with 
Corollary~\ref{C:TinF}. 
The asymptotic variance of $\cT_n$
stated in Corollary~\ref{C:TinF} 
can be calculated similarly, 
noting that ${\rm Var}(\cP)=\sum_{k=0}^\infty k^2\gamma_k$.

%%%%%%%%%%%%%%%%%%%%%%%%%%%%
%%%%%%%%%%%%%%%%%%%%%%%%%%%%
%%%%%%%%%%%%%%%%%%%%%%%%%%%%
%%%%%%%%%%%%%%%%%%%%%%%%%%%%
%%%%%%%%%%%%%%%%%%%%%%%%%%%%
\section{Cycle index asymptotics}
\label{S_final}

We conclude
with 
two related  applications
of \cite{HJ78,EH82}
to the asymptotics 
of cycle indices of 
the symmetric group.

%%%%%%%%%%%%%%%%%%%%%%%%%%%%
%%%%%%%%%%%%%%%%%%%%%%%%%%%%
%%%%%%%%%%%%%%%%%%%%%%%%%%%%
%%%%%%%%%%%%%%%%%%%%%%%%%%%%
%%%%%%%%%%%%%%%%%%%%%%%%%%%%
\subsection{Degrees in large random trees}
\label{S_deg}

Recall the cycle index $\zeta_n$, 
as defined in \eqref{E_zeta}.
Note that $\alpha^{-n}t(\alpha^n)\to t_1=1$, 
as $n\to\infty$. 
Therefore, 
by \eqref{E_LKZ} and   
Theorem 4.2 in 
\cite{HJ78} (where, for this application, we set 
$a_n=\alpha^{-n} t(\alpha^n)$
and 
$b_n=\alpha^{-n}\zeta_n$), 
it follows immediately that, as
$n\to\infty$, 
\begin{equation}\label{E_zeta_asy}
\zeta_n
\sim  e^\xi \alpha^n,
\end{equation}
where
\[
\xi=
\sum_{k=1}^\infty
\frac{1}{k}
\left(\frac{t(\alpha^k)}{\alpha^k}-1\right).
\]
Noting again that $t_1=1$, and 
reversing the order 
of summation,
we see that 
\[
e^\xi
=\prod_{k=2}^\infty
(1-
\alpha^{k-1})^{-t_k}.
\]

This result was proved by
Schwenk \cite{Sch77}, 
using an argument that involves
a certain generalization
of the inclusion--exclusion
principle. Another
proof was later given by 
Meir and Moon \cite{MM83}. 
In fact, as noted above, the result is 
a direct consequence of \cite{HJ78}, which precedes the latter.

The asymptotics
\eqref{E_zeta_asy}
are used, 
together with 
\cite{RS75}, 
to obtain information about 
degrees in large random trees; 
see Corollary 4.1 in \cite{Sch77}.

%%%%%%%%%%%%%%%%%%%%%%%%%%%%
%%%%%%%%%%%%%%%%%%%%%%%%%%%%
%%%%%%%%%%%%%%%%%%%%%%%%%%%%
%%%%%%%%%%%%%%%%%%%%%%%%%%%%
%%%%%%%%%%%%%%%%%%%%%%%%%%%%
\subsection{Sub-exponential sequences}
\label{S_CI}

Finally,  
by combining 
\cite[Theorem 1]{EH82}
and   \eqref{E_Zrec},
we observe the following
generalization of 
Corollary \ref{T_TF}. 

\begin{thm}
\label{T_CIasy}
Suppose that $(\nu_k,k\ge1)$ 
is a 
positive sequence, 
with a finite total sum 
$\lambda=\sum_{k=1}^\infty \nu_k<\infty$. 
Suppose also that $\mu_k=\nu_k/\lambda$
defines a  
sub-exponential probability 
distribution on the integers $k\ge1$. 
Then, as $n\to\infty$, 
\begin{equation}\label{E_Zn}
Z(S_n;\nu_1,2\nu_2,\ldots,n\nu_n)
\sim 
e^\lambda 
\nu_n.
\end{equation}
\end{thm}

In this work, 
we have demonstrated 
a procedure for 
obtaining the 
asymptotics of a sequence 
$(1=a_0,a_1,\ldots)$. 
That is, suppose, for some $\gamma>0$, 
that 
$\gamma^na_n$
is summable, so 
proportional 
to a probability distribution $\pi_n$. 
Then if 
\eqref{E_LK}
holds
for some 
positive sequence 
$(\nu_n,n\ge1)$, with total sum $\lambda$, and
$\mu_n=\nu_n/\lambda$ is sub-exponential, 
then $\pi_n\sim \nu_n$. 
The above theorem
reverses this procedure by, given $\nu_n$, 
identifying $\pi_n$ that satisfies (\ref{E_LK}). 
Indeed, for $\pi_n=e^{-\lambda}Z_n$ and
$Z_n$ the cyclic index in \eqref{E_Zn} above, 
(\ref{E_LK}) follows by setting
$x_i=i\nu_i$ in \eqref{E_Zrec} above. 
Thus, if $\mu_n=\nu_n/\lambda$
is sub-exponential, 
then $\pi_n\sim\nu_n$ 
and the theorem follows.

%%%%%%%%%%%%%%%%%%%%%%%%%%%%
%%%%%%%%%%%%%%%%%%%%%%%%%%%%
%%%%%%%%%%%%%%%%%%%%%%%%%%%%
%%%%%%%%%%%%%%%%%%%%%%%%%%%%
%%%%%%%%%%%%%%%%%%%%%%%%%%%%

\makeatletter
\renewcommand\@biblabel[1]{#1.}
\makeatother

\providecommand{\bysame}{\leavevmode\hbox to3em{\hrulefill}\thinspace}
\providecommand{\MR}{\relax\ifhmode\unskip\space\fi MR }
\providecommand{\MRhref}[2]{%
  \href{http://www.ams.org/mathscinet-getitem?mr=#1}{#2}
}
\providecommand{\href}[2]{#2}


\begin{thebibliography}{10}

\bibitem{Ald89}
D.~Aldous, \emph{Probability approximations via the {P}oisson clumping
  heuristic}, Applied Mathematical Sciences, vol.~77, Springer-Verlag, New
  York, 1989.

\bibitem{BG05}
A.~D. Barbour and B.~L. Granovsky, \emph{Random combinatorial structures: the
  convergent case}, J. Combin. Theory, Series A \textbf{109} (2005), no.~2,
  203--220.

\bibitem{BDK24a}
M.~Bassan, S.~Donderwinkel, and B.~Kolesnik, \emph{Graphical sequences and
  plane trees}, preprint at
  \href{https://arxiv.org/abs/2406.05110}{arXiv:2406.05110}.

\bibitem{BDK24b}
\bysame, \emph{Tournament score sequences, {E}rd{\H o}s--{G}inzburg--{Z}iv
  numbers, and the {L}\'evy--{K}hintchine method}, preprint at
  \href{https://arxiv.org/abs/2407.01441}{arXiv:2407.01441}.

\bibitem{BBCR00}
J.~A. Bell, E.~A. Bender, P.~J. Cameron, and L.~B. Richmond, \emph{Asymptotics
  for the probability of connectedness and the distribution of number of
  components}, Electron. J. Combin. \textbf{7} (2000), no.~R33, 1--22.

\bibitem{BS73}
R.~Bojanic and E.~Seneta, \emph{A unified theory of regularly varying
  sequences}, Math. Z. \textbf{134} (1973), 91--106.

\bibitem{Chi64}
V.~P. Chistyakov, \emph{A theorem on sums of independent random variables and
  its applications to branching random processes}, Theory Probab. Appl.
  \textbf{9} (1964), no.~4, 640--648.

\bibitem{CNW73}
J.~Chover, P.~Ney, and S.~Wainger, \emph{Functions of probability measures}, J.
  Analyse Math. \textbf{26} (1973), 255--302.

\bibitem{Com87}
K.~J. Compton, \emph{Some methods for computing component distribution
  probabilities in relational structures}, Discrete Math. \textbf{66} (1987),
  no.~1, 59--77.

\bibitem{DK24b}
S.~Donderwinkel and B.~Kolesnik, \emph{{S}ina\u{\i} excursions: An analogue of
  {S}parre {A}ndersen's formula for the area process of a random walk},
  preprint at \href{https://arxiv.org/abs/2403.12941}{arXiv:2403.12941}.

\bibitem{EH82}
P.~Embrechts and J.~Hawkes, \emph{A limit theorem for the tails of discrete
  infinitely divisible laws with applications to fluctuation theory}, J.
  Austral. Math. Soc. Ser. A \textbf{32} (1982), no.~3, 412--422.

\bibitem{EO84}
P.~Embrechts and E.~Omey, \emph{Functions of power series}, Yokohama Math. J.
  \textbf{32} (1984), no.~1-2, 77--88.

\bibitem{Fel68}
W.~Feller, \emph{An introduction to probability theory and its applications.
  {V}ol. {I}}, third ed., John Wiley \& Sons, Inc., New York-London-Sydney,
  1968.

\bibitem{Fin03}
S.~R. Finch, \emph{Mathematical constants}, Encyclopedia Math. Appl., vol.~94,
  Cambridge University Press, Cambridge, 2003.

\bibitem{FS09}
P.~Flajolet and R.~Sedgewick, \emph{Analytic combinatorics}, Cambridge
  University Press, Cambridge, 2009.

\bibitem{HJ78}
J.~Hawkes and J.~D. Jenkins, \emph{Infinitely divisible sequences}, Scand.
  Actuar. J. (1978), no.~2, 65--76.

\bibitem{Hey68}
C.~C. Heyde, \emph{On large deviation probabilities in the case of attraction
  to a non-normal stable law}, Sankhy\={a} Ser. A \textbf{30} (1968), 253--258.

\bibitem{Kol23}
B.~Kolesnik, \emph{The asymptotic number of score sequences}, Combinatorica
  \textbf{43} (2023), no.~4, 827--844.

\bibitem{Man24}
S.~R. Mane, \emph{Compound {P}oisson {D}istributions: {A} {G}eneral
  {F}ramework}, J. Indian Soc. Probab. Stat. \textbf{26} (2025), 277--338.

\bibitem{MM83}
A.~Meir and J.~W. Moon, \emph{On an asymptotic evaluation of the cycle index of
  the symmetric group}, Discrete Math. \textbf{46} (1983), no.~1, 103--105.

\bibitem{Mut11}
L.~Mutafchiev, \emph{Limit {T}heorems for the {N}umber of {P}arts in a {R}andom
  {W}eighted {P}artition}, Electron. J. Combin. \textbf{18} (2011), no.~1,
  \#P206, 1--27.

\bibitem{Nag69}
A.~V. Nagaev, \emph{Integral limit theorems taking large deviations into
  account when {C}ram{\'e}r's condition does not hold}, Theory Probab. Appl.
  \textbf{14} (1969), 51--64.

\bibitem{Ott48}
R.~Otter, \emph{The number of trees}, Ann. of Math. (2) \textbf{49} (1948),
  583--599.

\bibitem{PS79}
E.~M. Palmer and A.~J. Schwenk, \emph{On the number of trees in a random
  forest}, J. Combin. Theory Ser. B \textbf{27} (1979), no.~2, 109--121.

\bibitem{PR24}
K.~Panagiotou and L.~Ramzews, \emph{Asymptotic enumeration and limit laws for
  multisets: the subexponential case}, Ann. Inst. Henri Poincar\'{e} Probab.
  Stat. \textbf{60} (2024), no.~1, 612--635.

\bibitem{RS75}
R.~W. Robinson and A.~J. Schwenk, \emph{The distribution of degrees in a large
  random tree}, Discrete Math. \textbf{12} (1975), no.~4, 359--372.

\bibitem{Sch77}
A.~J. Schwenk, \emph{An asymptotic evaluation of the cycle index of a symmetric
  group}, Discrete Math. \textbf{18} (1977), no.~1, 71--78.

\bibitem{A000055}
N.~J.~A. Sloane, \emph{Number of trees with $n$ unlabeled nodes}, The
  {O}n-{L}ine {E}ncyclopedia of {I}nteger {S}equences, entry
  \href{https://oeis.org/A000055}{A000055}.

\bibitem{Stu23}
B.~Stufler, \emph{Probabilistic enumeration and equivalence of nonisomorphic
  trees}, preprint at
  \href{https://arxiv.org/abs/2305.16453}{arXiv:2305.16453}.

\bibitem{Stu20}
\bysame, \emph{Unlabelled {G}ibbs partitions}, Combin. Probab. Comput.
  \textbf{29} (2020), no.~2, 293--309.

\end{thebibliography}
\end{document}